\documentclass[a4paper]{elsarticle}
\usepackage{amssymb,amsmath,amsthm,amscd}
\usepackage{latexsym}
\usepackage{graphicx}
\usepackage{wrapfig}
\usepackage{arydshln}
\usepackage{enumerate}
\newtheorem{theorem}{Theorem}[section]
\newtheorem{lemma}[theorem]{Lemma}
\newtheorem{corollary}[theorem]{Corollary}
\newtheorem{proposition}[theorem]{Proposition}
\theoremstyle{definition}
\newtheorem{remark}[theorem]{Remark}
\newtheorem{definition}[theorem]{Definition}
\newtheorem{e}[theorem]{Example}
\newcommand{\sspan}[1]{\overline{span}\{#1\}}
\newcommand{\cran}[1]{\overline{\text{ran} (#1)}}
\newcommand{\ran}[1]{ran (#1)}

\begin{document}
\begin{frontmatter}
\title{On the diversity of twisted commuting operators}
\tnotetext[sup]{Research was supported by the Ministry of Science and Higher Education of the Republic
of Poland.}

\author{Zbigniew Burdak\corref{cor1}}
\ead{z.burdak@ur.krakow.pl}
\cortext[cor1]{Corresponding author}
\affiliation{organization={Department of Applied Mathematics, University of Agriculture},
            addressline={Balicka 253c},
            city={Krakow},
            postcode={30-198},
           country={Poland}}

\begin{keyword}
non commuting operators\sep q-commuting operators \sep twisted commuting operators \sep model of a pair of operators
\MSC[2020] 47A05 \sep 47B02
\end{keyword}

\begin{abstract}
Operators are called twisted commuting if the deformation from commutativity is determined by a unitary operator, including multiplication by a unimodular constant. The aim of the paper is to investigate constraints on the diversity of twisted commuting pairs depending on the deforming unitary called the twist. It turns out that some operators, like projections, are never twisted commuting. For many twists, twisted commutativity is possible only if at least one of the operators has a nontrivial kernel. In particular, bounded below operators, like isometries, are never twisted commuting by such twists. Twisted commuting unitaries are modeled by pairs of unitarily equivalent and commuting unitaries.
\end{abstract}
\end{frontmatter}
\section{Introduction}

The multivariable theory of Hilbert space operators is well understood for doubly commuting operator structures; there is an extensive theory for commuting systems, while significantly less is known if we drop the commutativity assumption. It is natural to investigate specific types of noncommutativity, where the deformation from commutativity is described. Perhaps, the simplest case is the anti-commuting case: $ST=-TS,$ first considered by Sebestien \cite{Sebestien}. The more general concept is rotation algebras (important in quantum theory), wherein the commutativity is deformed by a rotation factor, i.e., $ab=e^{2\pi i \theta}ba$ (introduced in \cite{Rieffel}; refer also to \cite{Boca,Palle,Web} and others). Such pairs are called $q$-commuting, where $q=e^{2\pi i \theta}$ is the deformation coefficient. A more general concept is to use a unitary operator as the deformation coefficient - Definition \ref{twis_def}. We call all of these concepts twisted commutativity. The subject attracts interest, as we can see in the recent contributions, such as: \cite{Ghatak,Majee, Majee2,Malik,Pinto,Pal,Rakish,Tom}.

The aforementioned results deal with extensions of (doubly) commuting case results to twisted (doubly) commuting case. Worth mentioning: Berger-Coburn-Lebow type models in \cite{BallSau,Ghatak}, decompositions - \cite{Majee, Majee2,Pal,Rakish}, Ando's type and regular dilations - \cite{BallSau,Ghatak,Sebestien, Malik,Tom}, among others. This paper, however, takes a different approach; we explore the diversity within the class of $U$ twisted commuting operators depending on the unitary $U$, referred to as the twist. It is obtained in Section 3. In Proposition \ref{U_uniq} we describe the uniqueness of the twist.  Next we describe diversity in edge cases understood as the richest and the purest diversity. As one could expect, the richest diversity is for the commuting case. In particular, twisted commuting idempotents are necessarily commuting. The purest diversity have unitaries that may twist solely pairs of product zero.

 Eigenspaces of a unitary reduce their commutants. Hence, eigenspaces of the twist reduce twisted commuting pairs to $q$-commuting pairs. Section 4 is devoted to this case.  The result, in particular, provides the model of twisted commuting pairs acting on finite dimensional spaces.

 Sections 3 and 4 demonstrate that twisted commuting operators with trivial kernels do not exist for many twists. In particular, pairs of $U$ twisted commuting, non-commuting isometries can exist for $U$ with specific properties. The results presented in \cite{BallSau,Ghatak} provide  Coburn-Berger-Lebow type models, which are model of pairs where the product is a unilateral shift, lacking a model of a pair of unitaries. To be precise, in \cite{BallSau}, the authors added to the model two $q$-commuting unitary operators. However, they are additional invariants, with no description. The existence of a pair of $U$ twisted commuting unitaries is not common. Recall that for the finite-dimensional case, the $q$-commuting unitaries exist only if $q=\root {n+1}\of 1$, where $n$ is the dimension. Section 5 investigates twisted commuting isometries. In particular pairs of twisted unitaries are described by pairs of commuting, unitarily equivalent unitaries.

We use the standard notation. The algebra of bounded, linear operators on a Hilbert space $H$ is denoted $B(H)$. The commutant of operator(s) is $\{\dots\}'$.  A spectrum, a point spectrum, a continuous spectrum, and an approximate point spectrum of $T$ are denoted $\sigma(T), \sigma_p(T), \sigma_c(T),\sigma_{ap}(T),$ respectively. The space of $\mathcal{H}$ valued, square summable Borel functions with complex domain is $L^2(\mu, \mathcal{H}),$ where the measure $\mu$ is a compactly supported Borel measure. If the space $\mathcal{H}$ is missing in the symbol, that is $L^2(\mu)$, we mean scalar-valued functions. If the measure is missing, it is assumed to be Lebesgue measure $m$ on the set given, usually unit circle $\mathbb{T}$. In particular, $L^2(\mathbb{T})$ and $H^2(\mathbb{T})$ are square summable functions on the unit circle $\mathbb{T}$ and its Hardy subspace, respectively. Similarly, $L^\infty(\mu)$ and $L^\infty(\mu, B(\mathcal{H}))$ are for scalar-valued and operator-valued respectively, Borel functions, bounded in $\mu$ essential supremum norm.

\section{Definition and basic properties}
This section is devoted to definition and basic properties of twisted commuting pairs.
\begin{definition}\label{twis_def}
A pair $(S, T)$ of bounded operators on $H$ is called twisted commuting if and only if there is a unitary operator $U\in \{S,T\}'$ such that $ST=UTS$.

The pair $(S, T)$ is called doubly twisted commuting if it is twisted commuting and $ST^*=U^*T^*S$.

In both cases, the operator $U$ is called the twist. The (doubly) twisted commuting pair with the twist $U$ is called $U$-(doubly) commuting in short ($q$-commuting if $U=q I$).
\end{definition}
Note that twisted commuting pairs are order sensitive: $(S, T)$ is $U$ twisted commuting if and only if $(T, S)$ is $U^*$ twisted commuting.
\begin{proposition}\label{gen_twisted}
We assume $(S, T)\in B(H)^2$ is a $U$-commuting pair for some unitary $U\in B(H)$.
Then:
\begin{enumerate}
\item $(S^*, T^*)$ is $U$-commuting,
\item $\ker S, \ker S^*, \ker T, \ker T^*$ are $U$ reducing,
\item $\|ST h\|=\|TS h\|$ for every $h\in H$,
\item $\ker S$ is invariant under $T$, $\ker T$ is invariant under $S$,
\item $\ker ST=\ker TS, \cran{ST}=\cran{TS}$ and both spaces are $U$ reducing,
\item $ST$ commutes with $TS.$
\end{enumerate}
\end{proposition}
\begin{proof}
Ad. (1) Since $U\in\{S,T\}'$ and $ST=UTS$ we get  $S^*T^*=UU^*S^*T^*=U(TSU)^*=U(UTS)^*=U(ST)^*=UT^*S^*.$

Ad. (2) By $U\in\{S,T\}'$ and Fuglede-Putnam theorem, it follows $U\in\{S^*,T^*\}'$ and, in turn, the result.

Ad. (3) It is immediate.

Ad. (4) It is a consequence of (3).

Ad. (5) It is a consequence of (3) and (1).

Ad. (6) Since $ST=UTS,$ and in turn $TS=U^*ST,$ we get $TSST=U^*STUTS=STTS$.
\end{proof}

If $(S,T)$ is $U$-doubly commuting, then by order sensitivity and Proposition \ref{gen_twisted}(1) the pair $(T, S^*)$ is $U$-commuting. In particular, Proposition \ref{gen_twisted} applies to $(T, S^*)$. However, for doubly twisted commuting pairs, additional properties also hold.
\begin{proposition}\label{gen_dtwisted}
If $(S, T)$ is a $U$-doubly commuting pair of bounded operators on $H,$ then:
\begin{enumerate}
\item $(S,T), (S^*, T^*), (T^*,S), (T, S^*)$ are $U$-commuting pairs,
\item $\ker S$ is reducing under $T$, $\ker T$ is reducing under $S$,
\item $T$ commutes with $S^*TS$, $S$ commutes with $T^*ST$.
\end{enumerate}
\end{proposition}
\begin{proof}
The first property follows directly from definition; the second follows from applications of Proposition \ref{gen_twisted}(4)  to pairs in (1). For the last property, we verify
$TS^*TS=TS^*U^*ST=S^*TUU^*ST=S^*TST$ and similarly for the other commutativity.
\end{proof}

Proposition \ref{gen_dtwisted}(3) may not hold for only twisted commuting pairs. As an example, take $S=T=T_z^*\in B(H^2(\mathbb{T})).$ The pair $(S,T)$ is commuting, non doubly commuting, and commutativity, as in Proposition \ref{gen_dtwisted}(3), does not hold.

Let us finish the section by recalling Coburn-Berger-Lebow type models of a pair of isometries.

\begin{theorem}\cite{BallSau}
A pair of $q$-commuting isometries is unitarily equivalent to the pair
\begin{align*}\left(\begin{array}{cc}R_q\otimes(I-P)U+M_zR_q\otimes PU& 0\\ 0& W_1\end{array}\right),\\ \left(\begin{array}{cc}R_{\bar{q}}\otimes U^*P+R_{\bar{q}}M_z\otimes U^*(I-P)& 0\\ 0& W_2\end{array}\right)\end{align*}
on $H^2(\mathbb{D},\mathcal{F})\oplus \mathcal{K}$ where $U\in B(\mathcal{F}), W_1, W_2\in B(\mathcal{K})$ are unitaries, $P\in  B(\mathcal{F})$ is a projection, and
 $R_q\in B(H^2(\mathbb{D}))$ is a rotation operator defined by $R_qf(z)=f(qz)$.
\end{theorem}

For the general twist, let us recall
\begin{theorem}\cite[Theorem 4.5]{Ghatak}
Let $(V,W)$ be a $U$-commuting pair of isometries, where the product $VW$ is a unilateral shift. Then $(V, W)$ is unitarily equivalent to $(M_{\phi_1}C_U,M_{\phi_2}C_{U^*})$ on $H^2(\mathbb{D},\ker (VW)^*),$
where \begin{align*}\phi_1(z)&=V|_{\ker W^*}\oplus zW^*|_{W(\ker V^*)},\\ \phi_2(z)&=W|_{\ker V^*}\oplus zU^*V^*|_{V(\ker W^*)},\\C_X\left(\sum_{n\ge 0}a_nz^nh_n\right)&=\sum_{n\ge 0}a_nz^nX^nh_n,\text{ for } X=U,U^*.\end{align*}
\end{theorem}

\section{Overall diversity constrains}
In this section, we explore the overall diversity constraints of the class of $U$-commuting pairs, depending on $U$. First of all, the class is never empty, as the pair $(0,I)$ is $U$-commuting for any unitary $U$. More generally, by Proposition \ref{gen_twisted}(5), a pair $(S,T)$ such that $ST=0$ is $U$-commuting for any $U\in\{S,T\}'.$ This raises the question about the ambiguity of the twist for a given pair of operators.

\begin{proposition}\label{U_uniq}
Let $(S, T) \in B(H)^2$ be a twisted commuting pair, and $H_0$ be the maximal subspace of $\ker ST$ that reduces both $S$ and $T.$  Then $H_0$ reduces each twist of $(S,T),$ and it is the maximal subspace where the twist is not unique. More precisely, there is a unique unitary $U_1\in B(H\ominus H_0)$ such that $ST_{H\ominus H_0}=U_1TS|_{H\ominus H_0}$ and the class of all the possible twists of $(S,T)$ is as follows $$\{U_0\oplus U_1: U_0 \text{ is a unitary on }H_0, U_0\in\{S|_{H_0}, T_{H_0}\}'\}.$$
\end{proposition}
\begin{proof}
By Proposition \ref{gen_twisted}(5), $U|_{\cran{TS}}\in B(\cran{TS})$. On the other hand, $ST=UTS$ defines $U$ on the linear manifold $TS(H),$ so by continuity, $U|_{\cran{TS}}$ is uniquely defined by $S$ and $T$. By Proposition \ref{gen_twisted}(1), we may apply the same argument to the pair $(S^*,T^*)$. Summing up, $U$ is uniquely defined on the subspace $\overline{\ran{ST}\vee \ran{S^*T^*}}=H\ominus(\ker ST\cap \ker S^*T^*)$ which reduces $U$. Let $H_1$ be the minimal subspace containing $\ran{ST}\vee \ran{S^*T^*}$ and reducing the pair $(S, T)$. Since $U$ doubly commutes with $S$ and $T,$ the subspace $H_1$ reduces $U$ as well, and it is uniquely defined on it. Hence $H_0=H\ominus H_1$ reduces $(S,T)$, and since it is a subspace of $\ker ST\cap\ker S^*T^*$, we get $ST|_{H_0}=TS|_{H_0}=0$. Consequently, any unitary commuting with $S|_{H_0}, T_{H_0}$ may be the twist of $(S|_{H_0}, T_{H_0})$.
\end{proof}

Note that pairs $(S,T)$ such that $ST=0$ are, in particular, commuting. Let us establish some notation:
\begin{definition}
A pair of operators is called trivially twisted commuting if it is (also) commuting, and nontrivially twisted commuting otherwise.

The class of $U$-commuting pairs is pure if the product of operators in each pair in the class is $0$.
\end{definition}

In particular, commuting pairs, that is, $I$-commuting pairs, are trivially twisted commuting, but the class of $I$-commuting pairs is not trivial.

Recall that commutant of a unitary is a commutative algebra if and only if its spectrum is simple (of spectral multiplicity one). Hence, all such unitaries (except for $I$ on $\mathbb{C}$) have a pure class of twisted commuting pairs. An example of such unitary is a bilateral shift of multiplicity one.  Since, by Fuglede-Putnam theorem, the commutant of a unitary operator is a $*$ - algebra, we may decompose a given $U$-commuting pair among $q$-commuting pairs (where $q$  are eigenvalues of $U$) and $U_c$-commuting pair, where $\sigma(U_c)=\sigma_c(U_c).$ In particular, by spectral multiplicity condition mentioned above,  if $q\neq 1$ is an eigenvalue of $U$ and the corresponding eigenspace has dimension one, then at least one of $U$-commuting operators vanishes on this eigenspace.

The general picture follows from the model of a unitary operator.
\begin{theorem}\cite[Ch. IX, Theorem 10.20]{Con}
If $U$ is a unitary operator on $H$, there are mutually singular
measures $\mu_\infty, \mu_1, \mu_2, \mu_3,\dots $ and an isomorphism
$$\mathcal{I}so:H\mapsto L^2(\mu_\infty,H_\infty)\oplus L^2(\mu_1,H_1)\oplus L^2(\mu_2,H_2)\oplus L^2(\mu_3,H_3)\oplus\dots,$$
such that
$$\mathcal{I}soU\mathcal{I}so^{-1}=U_\infty\oplus U_1\oplus U_2\oplus U_3\oplus\dots,$$
where $U_n=M_z\in B(L^2(\mu_n,H_n)),$ and $\dim H_n=n.$ Also
\begin{align*}\{U_\infty\oplus U_1\oplus U_2\oplus U_3\oplus\dots\}'&=L^\infty(\mu_\infty,B(H_\infty))\\ \oplus L^\infty(\mu_1,B(H_1))&\oplus L^\infty(\mu_2,B(H_2))\oplus L^\infty(\mu_3,B(H_3))\oplus\dots.\end{align*}

\end{theorem}

Clearly, $M_z\in B(L^2(\mu_1,H_1))$ has a simple spectrum, so its twisted commuting pairs are trivially twisted commuting. These are the only unitaries with a pure class of twisted commuting pairs. For operators acting on the vector-valued functions space (i.e., $\dim H_n$ at least two), the idea of the example comes from shift and clock matrices, that is the model of $q$-commuting unitaries on $n$-dimensional space, for $q=\root n+1 \of 1$:
\begin{equation}\label{mod}S=\left(\begin{array}{cccccc}
        0 & 0 &  \dots & \dots & 1 \\
        1 & 0 &  \ddots & \ddots &\vdots \\
       0 & 1 &  \ddots & \ddots & \vdots \\
        \vdots & \ddots  & \ddots & \ddots & \vdots \\
              0 & \dots  & 0 & 1 & 0 \\
      \end{array}
    \right),\quad T=\left(
      \begin{array}{cccccc}
       1 & 0 &  \ddots & \ddots & 0 \\
       0 & q &  \ddots & \ddots & \vdots \\
        \vdots & \ddots  & q^2  & \ddots & \vdots \\
       \vdots & \ddots  & \ddots & \ddots & 0 \\
        0 & 0 & \dots  & 0 & q^n  \\
      \end{array}
    \right).\end{equation}
The assumption $q=\root n+1 \of 1$ is required to obtain $(S,T),$ a pair of unitaries. However, if we replace $S$ with a truncated shift, we may consider any $q.$ More precisely, \begin{equation}\label{eq1}S=\left(
      \begin{array}{cccccc}
        0 & 0 &  \dots & \dots & 0 \\
        1 & 0 &  \ddots & \ddots &\vdots \\
       0 & 1 &  \ddots & \ddots & \vdots \\
        \vdots & \ddots  & \ddots & \ddots & \vdots \\
              0 & \dots  & 0 & 1 & 0 \\
      \end{array}
    \right),\quad T=\left(
      \begin{array}{cccccc}
       1 & 0 &  \ddots & \ddots & 0 \\
       0 & q &  \ddots & \ddots & \vdots \\
        \vdots & \ddots  & q^2  & \ddots & \vdots \\
       \vdots & \ddots  & \ddots & \ddots & 0 \\
        0 & 0 & \dots  & 0 & q^n  \\
      \end{array}
    \right).\end{equation}
    is $q$-commuting pair, where $S$ is a truncated shift in the case of finite-dimensional space or unilateral shift in the case of infinite-dimensional space.
The example generalizes to any unitary: let $S(z)=S$ (the constant function equal to a shift - truncated or unilateral depending on $\dim H_n$) and $T(z)=\text{diag}(1,z,z^2,z^3,\dots, z^n)$ for $n=\dim H_n<\infty$ or $T(z)=\text{diag}(1,z,z^2,z^3,\dots)$ for $\dim H_n=\infty$. Both functions are in $L^\infty(\mu, B(H_n)),$ so they commute with $M_z$. One may check, that $T(z)S=zST(z)$ for each $z,$ so $TS=M_zST$.
\begin{corollary}
Unitary operators equivalent to the operator of multiplication by the independent variable on the scalar-valued space of square integrable functions may be a twist of a trivial pair only; that is, a pair whose product is $0$.

A unitary operator equivalent to operator of multiplication by independent variable on a vector-valued space of square-integrable functions may be a twist of a nontrivial pair.\end{corollary}

The matrices in the previous example may be operator-valued. More precisely, if $\mathcal{U}\in B(\mathcal{H})$ is any unitary, then
$$S=\left(
      \begin{array}{cccccc}
        0 & 0 &  \dots & \dots & 0 \\
        \mathcal{I} & 0 &  \ddots & \ddots &\vdots \\
       0 & \mathcal{I} &  \ddots & \ddots & \vdots \\
        \vdots & \ddots  & \ddots & \ddots & \vdots \\
              0 & \dots  & 0 & \mathcal{I} & 0 \\
      \end{array}
    \right),\quad T=\mathcal{I} \oplus \mathcal{U} \dots \mathcal{U}^{n-1}$$
 is $U=\mathcal{U}\oplus \mathcal{U}\oplus\dots\oplus \mathcal{U}$-commuting pair on $\mathcal{H}\oplus \mathcal{
H}\oplus\dots\oplus \mathcal{H}.$ In particular, for $\mathcal{U}$ a bilateral shift of multiplicity one, we get $U$ a bilateral shift of higher multiplicity.

In each of the above examples, one operator has a nontrivial kernel. We will see in the next section, that for many unitaries $U,$ the class of $U$-commuting pairs is not pure, but consists of pairs with at least one nontrivial kernel. In particular, those unitaries may not twist a pair of isometries. The positive examples follow from \eqref{mod} or Pauli matrices.

    So far, we have seen, that the class of $U$-commuting pairs is never empty; its size and variety depend on $U$. Let us finish the section with the result that imposes a limitation on the class of nontrivially twisted commuting pairs: projections are never in such a class.
    \begin{remark}
Let $(P,S)$  be $U$-commuting, where $P$ is $n$ - potent ($P^n=P$). Then $UPS=SP=SP^n=UPSP^{n-1}=\dots=(UP)^nS=U^nP^nS=U^nPS,$ and since $U$ is unitary, we get $PS=U^{n-1}PS$. Let $U=U_0\oplus U_1$ be a decomposition as in Proposition \ref{U_uniq}. Then $U_1$ is $n-1$-potent itself, and taking  $U_0=I$, the whole $U$ turns out to be $n-1$ potent.

Hence, $n$-potent may be $U$-commuting provided the spectrum of $U$ restricts to a point spectrum, and consists of $n-1$-th roots of $1$. In particular, idempotent may only trivially twisted commute; twisted commuting pairs having a $3$-potent may be decomposed into commuting and anti-commuting pairs. Indeed, if $P$ is a $3$-potent, then $U$ is an idempotent unitary, so it has two eigenvalues $-1$ and $1$.

In particular projections may not be nontrivially twisted commuting.
\end{remark}

\section{Diversity and model of $q$-commuting pairs}
The section focuses on $q$-commuting pairs of operators. In the paper \cite{Bat} the authors provide the model of $q$-commuting tuples using the $q$-commuting shift. We are going to give a different description.
 It is natural to assume nontrivially twisted commuting pairs, so $q\neq 1$.
From now on, we assume the twist $U=q I,$ so a $q$-commuting pair, where $q\neq 1$.

First we investigate the approximate point spectrum of operators in $q$-commuting pairs. Since the approximate point spectrum contains the border of a spectrum, we gain information on the shape of the spectrum. Much more can be said about point spectrum, which is investigated in the following part.
\begin{lemma}\label{re_aps}
Let $(S,T)$ be $q$-commuting pair, where $q\in\mathbb{T}\setminus\{1\}.$

If $T$ is bounded below, then $\sigma_{ap}(S)$ is invariant under multiplication by $q$.

In particular, if $q=e^{i\theta}$ where $\theta/\pi\notin \mathbb{Q},$ then $\sigma_{ap}(S)$ is invariant under rotation around the origin.

\end{lemma}
\begin{proof}
If $0\neq\alpha\in\sigma_{ap}(S)$, then there exists $\{x_n\}$ such that $\|x_n\|=1$ and $\|Sx_n-\alpha x_n\|\mapsto 0.$ Note that   $$\|STx_n- q\alpha Tx_n\|=\| q TSx_n- q\alpha Tx_n\|=\| q T(Sx_n-\alpha x_n)\|\mapsto 0.$$ If $T$ is bounded below, then $\sup\{1/\|Tx_n\|\}=M<\infty$ and for $y_n=Tx_n/\|T x_n\|$ we have $\|Sy_n- q\alpha y_n\|\leq M\|STx_n- q\alpha Tx_n\|\mapsto 0$. Hence, $ q\alpha\in\sigma_{ap}(S)$.

Since the set $\{e^{in\theta}\alpha\}_{n\ge 0}$ is dense in the circle of radius $|\alpha|$ for $\theta/\pi\notin \mathbb{Q}$ and the approximate point spectrum is closed, we obtain the last part.
\end{proof}
The simple example where $T$ is not bounded below and the result does not hold is the following anti-commuting pair $$S = \begin{pmatrix} 1 & 0 & 0 \\ 0 & -1 & 0 \\ 0 & 0 & 2 \end{pmatrix}, \quad  T=\begin{pmatrix} 0 & 1 & 0 \\ 1 & 0 & 0 \\ 0 & 0 & 0 \end{pmatrix}.$$

Since the border of the spectrum is in the approximate point spectrum, under the assumption of Lemma \ref{re_aps}, we get the shape of the spectrum.
\begin{corollary}
Let $(S,T)$ be $q$-commuting pair, where $ q\in\mathbb{T}\setminus\{1\}$ and $T$ is bounded below. If  $ q=e^{i\theta}$ where $\theta/\pi\notin \mathbb{Q}$, then $\sigma(S)$ is invariant under rotation around the origin.
\end{corollary}
\begin{proof}
Indeed, if we take a joint component of the spectrum $\alpha\in \sigma(S),$ then either the circle $\{z:|z|=|\alpha|\}\subset \sigma(S)$ or there is $\alpha_{ap}\in\{z:|z|=|\alpha|\}\cap\partial\sigma(S)\subset\sigma_{ap}(S)$. In the latter case, by Lemma \ref{re_aps}, the whole circle of radius $|\alpha_{ap}|=|\alpha|$ is a subset of $\sigma_{ap}(S)\subset\sigma(S)$. Hence, the spectrum is invariant under the rotation.
\end{proof}
In particular, for pairs of isometries, we get:

\begin{corollary}\label{sp_unit}
Let $(S,T)$ be a $q$- commuting pair of isometries, where $ q\in\mathbb{T}\setminus\{1\}$.
Then $\sigma(S)$ and $\sigma(T)$ are closed under multiplication by $ q$.

If $ q$ is not a root of $1$ of any multiplicity, then each of $\sigma(S)$ and $\sigma(T)$ is either the whole disk or the whole circle.

In particular, if $ q$ is not a root of $1$, then each unitary in a  $ q$- commuting pair of unitaries has the spectrum $\mathbb{T}.$
\end{corollary}

Much more can be said about point spectrum of operators in $q$-commuting pairs. More precisely, we give the model of such pairs on the spaces generated by generalized eigenspaces corresponding to one eigenvalue of a given operator.

The generalized eigenspace of an operator is invariant for its commutant. In fact, it is a special case of a more general result described by Lemma \ref{re1} for $q$-commuting pairs. The operator $T,$ which is $q$-commuting with $S,$ maps the generalized eigenspace of $S$ and eigenvalue $\alpha$ to a generalized eigenspace of $S$ and eigenvalue $ q\alpha.$  Hence, a generalized eigenspace of $S$ is invariant under $T,$ if it is the eigenspace of $0$ or $T$ vanishes on it (we excluded $ q=1$). Denote by $J_{S,\alpha}$ the part of Jordan normal matrix of an operator $S\in B(H)$ and eigenvalue $\alpha$ (possibly consisting of several Jordan blocks) and by $H_{S,\alpha}$ the corresponding generalized eigenspace. Let us start with an example.

\begin{e}\label{ex_mod}
Let $ q\in\mathbb{T}, \alpha\in\mathbb{C}\setminus\{0\}$ and $\{v_1,\dots, v_8\}$ be a basis of $\mathbb{C}^8$. Define $S$ by a matrix in this basis
\begin{equation}\label{ex2}S=\left(\begin{array}{ccc|cc|ccc} \alpha & 1 & 0 & 0 & 0 & 0 & 0 & 0 \\
0 & \alpha & 1 & 0 & 0 & 0 & 0 & 0  \\ 0 & 0 & \alpha  & 0 & 0 & 0 & 0 & 0 \\\hline 0 & 0 & 0 &  q\alpha & 1 & 0 & 0 & 0\\ 0 & 0 & 0 & 0 &  q\alpha & 0 & 0 & 0\\\hline 0 & 0 & 0 & 0 & 0 & q^2\alpha & 1 & 0\\
0 & 0 & 0 & 0 & 0 & 0 &  q^2\alpha & 1 \\0 & 0 & 0 & 0 & 0 & 0 & 0 &  q^2\alpha\end{array}\right).\end{equation}

Recall, that we assumed $ q\ne 1$. Hence, either $ q^2\neq 1$ or $ q=-1$.

If $ q^2\neq 1,$ then $\alpha, q\alpha,$ and $q^2\alpha$ are distinct eigenvalues, and \eqref{ex2} corresponds to Jordan form:
\begin{equation}\label{ex2}S=\left(\begin{array}{c|c|c} &&\\J_{S,\alpha} & 0 & 0 \\&&\\
\hline \begin{array}{c} \text{ }\\\text{ } \end{array} 0 \begin{array}{c} \text{ }\\\text{ } \end{array}&  J_{S, q\alpha} & 0\\\hline &&\\0 & 0 & J_{S, q^2\alpha}\\&&\end{array}\right)=\left(\begin{array}{ccc|cc|ccc} \alpha & 1 & 0 & 0 & 0 & 0 & 0 & 0 \\
0 & \alpha & 1 & 0 & 0 & 0 & 0 & 0  \\ 0 & 0 & \alpha  & 0 & 0 & 0 & 0 & 0 \\\hline 0 & 0 & 0 &  q\alpha & 1 & 0 & 0 & 0\\ 0 & 0 & 0 & 0 &  q\alpha & 0 & 0 & 0\\\hline 0 & 0 & 0 & 0 & 0 & q^2\alpha & 1 & 0\\
0 & 0 & 0 & 0 & 0 & 0 &  q^2\alpha & 1 \\0 & 0 & 0 & 0 & 0 & 0 & 0 &  q^2\alpha\end{array}\right),\end{equation}
and $v_1,\dots,v_8$ are generalized eigenvectors, and  $H_{S,\alpha}=\bigvee\{v_1, v_2, v_3\}, H_{S, q\alpha}=\bigvee\{v_4, v_5\}, H_{S, q^2\alpha}=\bigvee\{v_6, v_7, v_8\}.$

Then $$T=\left(\begin{array}{ccc|cc|ccc}
0 & 0 & 0 & 0 & 0 & 0 & 0 & 0 \\
0 & 0 & 0 & 0 & 0 & 0 & 0 & 0  \\
0 & 0 & 0 & 0 & 0 & 0 & 0 & 0 \\\hline
0 & a & b & 0 & 0 & 0 & 0 & 0\\
0 & 0 &  q a& 0 & 0 & 0 & 0 & 0\\\hline
0 & 0 & 0 & c & d & 0 & 0 & 0\\
0 & 0 & 0 & 0 &  q c & 0 & 0 & 0 \\
0 & 0 & 0 & 0 & 0 & 0 & 0 & 0\end{array}\right)$$

is $q$-commuting with $S$ operator for any $a,b, c, d$ (including some of them, or all equal to $0$).

If $ q=\root 3\of 1$ the formula on $T$ is more general:
 $$T=\left(\begin{array}{ccc|cc|ccc}
0 & 0 & 0 & 0 & 0 & e & f & g \\
0 & 0 & 0 & 0 & 0 & 0 &  q e&  q f  \\
0 & 0 & 0 & 0 & 0 & 0 & 0 &  q^2 e\\\hline
0 & a & b & 0 & 0 & 0 & 0 & 0\\
0 & 0 &  q a& 0 & 0 & 0 & 0 & 0\\\hline
0 & 0 & 0 & c & d & 0 & 0 & 0\\
0 & 0 & 0 & 0 &  q c& 0 & 0 & 0 \\
0 & 0 & 0 & 0 & 0 & 0 & 0 & 0\end{array}\right).$$
for any $a,b, c, d, e, f, g.$

If   $ q=-1$ the block $J_{S, q^2\alpha}$ becomes the second Jordan block in $J_{S,\alpha}.$ In the reordered basis $\{v_1, v_2, v_3, v_6, v_7, v_8,v_4,v_5\},$ we get the matrix
\begin{equation}\label{ex21}S=\left(\begin{array}{ccc|c} &&&\\&&&\\&J_{S,\alpha} &  & 0 \\&&&\\&&&\\\hline
\begin{array}{c} \text{ }\\\text{ } \end{array}&  0 & \begin{array}{c} \text{ }\\\text{ } \end{array}&  J_{S, q\alpha} \\\end{array}\right)=\left(\begin{array}{ccc:ccc|cc} \alpha & 1 & 0 & 0 & 0 & 0 & 0 & 0 \\
0 & \alpha & 1 & 0 & 0 & 0 & 0 & 0  \\ 0 & 0 & \alpha  & 0 & 0 & 0 & 0 & 0 \\\hdashline 0 & 0 & 0 & \alpha & 1 & 0 & 0 & 0\\ 0 & 0 & 0 & 0 & \alpha & 1 & 0 & 0\\ 0 & 0 & 0 & 0 & 0 &\alpha & 0 & 0\\\hline
0 & 0 & 0 & 0 & 0 & 0 & -\alpha & 1 \\0 & 0 & 0 & 0 & 0 & 0 & 0 & -\alpha\end{array}\right),\end{equation}

and the corresponding formula on $q$-commuting with $S$ operator $$T=\left(\begin{array}{ccc:ccc|cc}
0 & 0 & 0 & 0 & 0 & 0 & g & h \\
0 & 0 & 0 & 0 & 0 & 0 & 0& -g  \\
0 & 0 & 0 & 0 & 0 & 0 & 0 & 0\\\hdashline
0 & 0 & 0 & 0 & 0 & 0 & e & f\\
0 & 0 & 0 & 0 & 0 & 0 & 0 & - e\\
0 & 0 & 0 & 0 & 0 & 0 & 0 & 0\\\hline
0 & a & b & 0& c & d  & 0 & 0 \\
0 & 0 & -b & 0& 0 & -c  & 0 & 0\end{array}\right).$$
\end{e}

We can see in this example that  $T$ maps the generalized eigenspace of $\alpha$ to the generalized eigenspace of $ q\alpha$ and the latter one to the generalized eigenspace of $ q^2\alpha$. In the case $ q^3=1,$ the rule extends to the pair $ q^2\alpha, \alpha$. If $ q=-1$ we have two Jordan blocks in $J_{S,\alpha}$ and $T$ maps between eigenspaces of $\alpha$ and $-\alpha$. Note that there is no relation between dimensions of the corresponding eigenspaces, and $T$ is neither injective nor surjective in general. If $ q\neq\root 3\of 1,$ then $T$ has to vanish on the eigenspace of $ q^2\alpha$, as it maps to the trivial eigenspace of $ q^3\alpha$. Hence, in the case of finite-dimensional space, only for certain values of $q$ (precisely, roots of $1$ of the  respective order) there are $q$-commuting pairs with trivial kernels. For isometries (unitaries) we get \eqref{mod}.

The above observation is a general rule.

\begin{lemma}\label{re1}
Let $(S,T)$ be a $q$-commuting pair, where $ q\in\mathbb{T}\setminus\{1\}$, and $\alpha$ be an eigenvalue of $S$. Then $$T:H_{S,\alpha}\mapsto H_{S, q\alpha}.$$ Moreover, if $(S,T)$ are q-doubly commuting, then \begin{align*}
T^*:H_{S,\alpha}\mapsto H_{S,\bar{ q}\alpha}\\(T|_{H_{S,\alpha}})^*=T^*|_{H_{S, q\alpha}}.\end{align*}
\end{lemma}
Note, that the statement does not mean that $ q\alpha$ is an eigenvalue of $S,$ as it is possible that $T|_{H_{S,\alpha}}=0$.
\begin{proof}
The set of generalized eigenvectors forms a basis of $H_{S,\alpha}$. Let us consider a full sequence of generalized eigenvectors, that is, vectors corresponding to one Jordan block: $x_1,\dots,x_k$  such that $(S-\alpha I)x_i=x_{i-1}\neq 0, (S-\alpha I)x_1=0$. Note that, \begin{equation}\label{r1}STx_i= q TSx_i= q T(\alpha x_i+x_{i-1})= q\alpha Tx_i+ q Tx_{i-1}\text{ for }i=2,\dots,k,\end{equation} and $$STx_1= q TSx_1= q \alpha Tx_1.$$

By \eqref{r1}, if $Tx_i=0$, then $Tx_{i-1}=0$. Let $0\leq l\leq k$ be the maximal number such that $Tx_l= 0,$ where $l=0$ if $Tx_1\neq 0$. If $l=k$, then trivially $Tx_i=0\in H_{S, q\alpha}$. If $l< k$, then  $STx_{l+1}=\alpha q Tx_{l+1}\neq 0$ yields that $\alpha q$ is an eigenvalue of $S$. Moreover,  $Tx_{l+1},\bar q Tx_{l+2},\dots,\bar q^{k-l-1}Tx_k$  is (not necessarily full) set of generalized eigenvectors of $J_{S,\alpha q}.$ It may be supplemented to the full sequence
\begin{equation}\label{ee}Tx_{l+1},\bar q Tx_{l+2},\dots,\bar q^{k-l-1}Tx_k,y_1,\dots,y_j,\end{equation}
which is in $H_{S,\alpha q}$. Hence we get
$T:H_{S,\alpha}\mapsto H_{S, q\alpha}.$
More precisely, $T$ maps generalized eigenvectors corresponding to one Jordan block of $\alpha$ to generalized eigenvectors corresponding to one Jordan block of $ q\alpha$.

If $(S,T)$ are $q$-doubly commuting, then  $(S,T^*)$ are $\bar{q}$-commuting which implies the remaining part of the statement.

\end{proof}

It is known, that the generalized eigenspace is hyperinvariant, thus reducing for doubly commuting operators. In the case of $q$-commutativity, the relation is more complicated. Clearly, by $q$-hyperinvariant subspace, we mean a subspace invariant under any $q$-commuting operator.
\begin{corollary}
Let $\alpha$ be an eigenvalue of $S$. Then $\bigvee_{n\ge 0}H_{S,q^n\alpha}$ is \linebreak $q$-hyperinvariant, where $H_{S,q^n\alpha}$ are generalized eigenspaces of $q^n\alpha$. In particular generalized eigenspace of $0$ is $q$-hyperinvariant.\end{corollary}

 In particular, it implies the model of $q$-doubly commuting pair on finite-dimensional spaces. Indeed, doubly $q$-commuting pair may be decomposed among pairs as in the following corollaries. The corollaries are under weaker that doubly commutativity assumptions.

\begin{corollary}\label{model}
Let $(S,T)$ be a $q$-commuting pair, $\alpha\neq 0$ be an eigenvalue of $S,$ and $ q^i\neq 1$ for $i=1,\dots,n-1,$ and $$S=\left(\begin{array}{ccccc} J_{S,\alpha} & 0 & \cdots & \cdots & 0 \\ 0 & J_{S, q\alpha} &\ddots & \ddots & \vdots\\ \vdots & \ddots & \ddots& \ddots& \vdots\\ \vdots & \ddots & \ddots & \ddots & 0\\ 0 & \cdots & \cdots & 0 & J_{S, q^{n-1}\alpha} \end{array}\right).$$
\begin{itemize}
\item If $ q^n\neq 1$,  then neither of $\bar q\alpha,  q^n\alpha$ is an eigenvalue of $S,$ and $T$ is of the form
$$T=\left(\begin{array}{ccccc} 0 & \cdots & \cdots & \cdots & 0 \\ T_1 & \ddots & \ddots & \ddots & 0\\ 0 & T_2 & \ddots & \ddots & 0\\ \vdots & \ddots & \ddots & \ddots & 0\\ 0 & \cdots & 0 & T_{n-1} & 0 \end{array}\right).$$
\item If $ q^n=1,$
 then $T$ has a more general form $$T=\left(\begin{array}{ccccc} 0 & \cdots & \cdots & 0 & T_n \\ T_1 & \ddots & \ddots & \ddots & 0\\ 0 & T_2 & \ddots & \ddots & 0\\ \vdots & \ddots & \ddots & \ddots & 0\\ 0 & \cdots & 0 & T_{n-1} & 0 \end{array}\right).$$
 If $T_i=0,$ then we may reorder the Jordan parts of $S$ to \linebreak $J_{S, q^i\alpha},\dots,J_{S, q^{n-1}},J_{S,\alpha}\dots, J_{S, q^{i-1}\alpha}.$ Then $T_i=0$ is positioned in upper right corner of $T$ matrix.
\end{itemize}

Moreover, if there is a unique Jordan block of a given eigenvalue, then taking the basis \eqref{ee} as in Lemma \ref{re1}, we get the form \begin{equation}\label{Tiform}T_i=\left(\begin{array}{cc} 0 & \text{diag}(1, q,\dots, q^{k-l-1})\\ 0 & 0\\\end{array}\right),\end{equation} where $k-l-1$ is as in Lemma \ref{re1} and $0$ are matrices of the respective size. In particular, $\ker T_i=\{0\}$ if and only if $l=0$ and the matrix of $T_i$ is square, so the corresponding Jordan blocks of $ q^{i-1}\alpha$ and $ q^{i}\alpha$ are of the same size.
\end{corollary}

Note that assumptions on the form of $S$ in Corollary \ref{model} yields $\alpha,\dots,q^{n-1}\alpha$ are the only eigenvalues of $S$. The similar result is possible for infinite number of eigenvalues.
\begin{corollary}\label{model1}
Let $(S,T)$ be $q$-commuting pair, $\alpha\neq 0$ be an eigenvalue of $S,$ and  $ q^i\alpha$ are different eigenvalues of $S$ ($ q^i\neq 1$) for all $i\in\mathbb{Z}_+,$
and $$S=\left(\begin{array}{cccc} J_{S,\alpha} & 0 & \cdots &\cdots  \\ 0 & J_{S, q\alpha} &\ddots & \ddots \\ \vdots & \ddots & \ddots& \ddots\\ \vdots & \ddots & \ddots& \ddots \end{array}\right).$$ Then
$$T=\left(\begin{array}{cccc}  0 & \cdots &\cdots &\cdots \\ T_1 &0  &\ddots & \ddots \\ 0 & T_2 & \ddots& \ddots\\ \vdots & \ddots & \ddots& \ddots \end{array}\right).$$
Moreover, the form of $T_i$ as in \eqref{Tiform} remains correct under the same assumptions.

The similar results are correct, if  $ q^i\neq 1$  and $ q^i\alpha$ are different eigenvalues of $S$ for all $i\in\mathbb{Z}_-$ or  all $i\in\mathbb{Z}$.
\end{corollary}
The next result summarizes the limitations on $q$-commuting pairs with trivial kernels, in particular $q$-commuting isometries.

\begin{corollary}Let $S,T$ be a $q$-commuting pair, where $q\neq 1$.
\begin{itemize}
\item If $ q^{k-1}\alpha\neq 0$ is an eigenvalue of $S$ and $ q^{k}\alpha$ is not an eigenvalue of $S,$ then $T$ vanishes on $H_{S, q^{k-1}\alpha}$.
\item If $\dim H<\infty$ and $\alpha\neq 0$ is an eigenvalue of $S,$ then $\ker T=\{0\}$ implies that $ q$ is a root of $1$ of arbitrarily small degree, and the Jordan blocks of eigenvalues $ q^i\alpha$ are of the same size.
\end{itemize}
\end{corollary}

The generalized eigenspace of $S$ does not necessarily reduce $S$, and similarly, there is no reason for $\bigvee_n  H_{S,q^n\alpha}$ to reduce $S$. This is true for normal operators, and we formulate the following Theorem \ref{mod_thm} for normal operator $S,$ which is reduced by a generalized eigenspace, and in turn, $\bigvee_n  H_{S,q^n\alpha}$ reduces $q$-commuting pair (in fact, $q$-doubly commuting). However, the above result works for a more general case. Assume there are eigenvalues $\{\alpha_i\}_{i\in J}$  such that $\alpha_i\neq q^n\alpha_j$ for $i\neq j$ and any $n$, and the corresponding spaces $\{\bigvee_n  H_{S,q^n\alpha_i}\}_{i\in  J}.$ Any $T$ that $q$-commutes with $S$ is described on each of such spaces by Lemma \ref{re1}. Hence, we get a description of $q$-commuting pairs in the case where the spectrum of one of operators is equal to a point spectrum. In particular, we obtain a model on finite-dimensional spaces. The general description of such a pair is technically complicated. Indeed, if there are many $\alpha_i$ as in Corollary \ref{model1} with infinite sequences $\alpha_1, q\alpha_i,\dots$, then we need to interlace different blocks $J_{S,q^m\alpha_i}$ and $J_{S,q^n\alpha_j}$ in $S$ matrix and place the corresponding blocks in specific locations in $T$ matrix. Hence, instead of a general form, we present the idea on  the relatively small spaces $\mathbb{C}^2$ and  $\mathbb{C}^3.$
\begin{theorem}\label{mod_thm}
Let $(S,T)$ be a $q$-commuting pair, where $S$ is a normal and $q\neq 1$. If $\alpha$ is an eigenvalue of $S$ (or $T$) and $H_\alpha$ is the minimal subspace reducing $(S,T)$ containing the eigenspace $H_{S,\alpha},$ then  $(S|_{H_\alpha},T|_{H_\alpha})$ takes one of the forms outlined in Corollary \ref{model} or Corollary \ref{model1}.
\end{theorem}
\begin{proof}
For a given $\alpha,$ an eigenvalue of $S,$ there is a maximal sequence \linebreak $ q ^i\alpha , q ^{i+1}\alpha,\dots  q^j\alpha$ of different eigenvalues of $S$ for some $i\leq 0\leq j,$ where the maximality means that either $ q ^{i-1}= q ^j$ or: ($i=-\infty$ or $ q ^{i-1}\alpha$ is not an eigenvalue of $S$) and ($j=\infty$ or $ q ^{j+1}\alpha$ is not an eigenvalue of $S$). Then, by Lemma \ref{re1}, the space $\bigvee_{\iota=i}^jH_{S, q^\iota\alpha}$ reduces $T$, so it reduces the pair. Moreover, the restriction has one of the forms in Corollary \ref{model} and Corollary \ref{model1}.
\end{proof}

\subsection{Model of $q$-commuting pairs on finite-dimensional spaces}

The results on the point spectrum clearly describe the $q$-doubly commuting pairs on finite dimensional spaces. However, it describes also the case of only $q$-commuting operators. In this subsection, we provide details on relatively small subspaces $\mathbb{C}^2$ and $\mathbb{C}^3,$ which illustrate the general idea. In general, the twist may have more than one eigenvalue. If this is the case, the pair decomposes into pairs on lower dimensional spaces - eigenspaces of the twist. Hence, we restrict to the twist of the form $q I,$ thus considering $q$-commuting pairs. For the sake of completeness, we should mention the trivial one-dimensional case: if $ q\neq 1$ then the pair has product $0$ so one of the operators is trivial. In the remaining part we are looking for  nontrivially $q$-commuting pairs, so we assume $q\neq 1$ and describe pairs, such that $ST\neq 0$.

The case $\mathbb{C}^2.$

\begin{remark}\label{C2}
Assume $S, T$ on $\mathbb{C}^2$ are $q$-commuting, where $ q\neq 1$ and $ST\neq 0.$  Then, there can be chosen a basis such that one of the following holds:
$$S=\left(
      \begin{array}{cc}0 & 1\\0 &0\\\end{array}\right),\quad T=\left(
      \begin{array}{cc}a & b\\0 & q a\\\end{array}\right),$$
where $a\neq 0$ is necessary for $ST\neq 0,$ or
 $$S=\left(
      \begin{array}{cc}\alpha & 0\\0 & q \alpha\\\end{array}\right),\quad T=\left(\begin{array}{cc}0 & 0\\b &0\\\end{array}\right),$$
or if $ q=-1,$ then
$$S=\left(
      \begin{array}{cc}\alpha & 0\\0 &- \alpha\\\end{array}\right),\quad T=\left(\begin{array}{cc}0 & a\\ b &0\\\end{array}\right)$$
in the respective basis.
\end{remark}
\begin{proof}
If $0$ is an eigenvalue of $S,$ then one of the cases holds:
\begin{itemize}
\item $S=\left(
      \begin{array}{cc}0 & 1\\0 &0\\\end{array}\right)$ and one can check by a direct calculation that $T$ is of the form as in the statement,
\item $S=0$ so it is not the case $ST\neq 0,$
\item There is the second eigenvalue $\alpha\neq 0$ of $S,$ so $S(ax+by)=\alpha by,$ where $x$ is the eigenvector of $0$ and $y$ is the eigenvector of $\alpha$. By Lemma \ref{re1}, $T$ maps $H_{S,\alpha}=\mathbb{C}y$ to $H_{S, q\alpha}=\{0\},$ so  $T(y)=0,$ and in turn $TS(ax+by)=0,$ which is not the case.
    \end{itemize}

     If $S$ has two non-zero eigenvalues $\alpha_1, \alpha_2$, then by Lemma \ref{re1}, $T$ maps $H_{S,\alpha_i}$ to $H_{S, q\alpha_i}$ for $i=1,2$. If there is neither $\alpha_1= q\alpha_2$ nor the other way, then $T$ maps both eigenspaces to trivial spaces, so $T=0$ which is not the case.
     If there is $\alpha_1= q\alpha_2$, we get the second or the third case by Corollary \ref{model}, respectively.
\end{proof}
 The case  $\mathbb{C}^3.$

 \begin{remark}\label{C3}
 Assume $S, T$ on $\mathbb{C}^3$ are $q$-commuting, where $ q\neq 1$ and $ST\neq 0.$  Then, there can be chosen a basis such that one of the following holds:
 \begin{enumerate}
 \item There is a basis $\{x,y,z\}$ such that $$S(ax+by+cz)=\alpha ax+S_1(by+cz)$$ and $$T(ax+by+cz)=T_1(by+cz)$$  where the pair $(S_1, T_1)$ is unitarily equivalent to one of the pairs in Remark \ref{C2},
 \item There is a basis $\{x,y,z\}$ such that $$S(ax+by+cz)=S_1(by+cz)$$ and $$T(ax+by+cz)=aT(x)+T_1(by+cz)$$  where the pair $(S_1, T_1)$ is unitarily equivalent to one of the pairs in Remark \ref{C2},
 \item $$S=\left(
      \begin{array}{ccc}0 & 1 & 0\\0 & 0 & 1\\0 & 0 & 0\end{array}\right),\quad T=\left(
      \begin{array}{ccc}a & b & c\\0 & q a &  q b\\ 0 & 0 &  q^2 a\\\end{array}\right)$$
where $a\neq 0$ or $b\neq 0,$
\item      $$S=\left(
      \begin{array}{ccc}0 & 0 & 0\\0 & 0 & 1\\0 & 0 & 0\end{array}\right),\quad T=\left(
      \begin{array}{ccc}a & 0 & c\\d & b & f\\ 0 & 0 &  q b\\\end{array}\right)$$
      where $b\neq 0,$
 \item  $$S=\left(
      \begin{array}{ccc}\alpha & 0 & 0\\0 &  q\alpha & \delta\\0 & 0 &  q\alpha\end{array}\right),\quad T=\left(
      \begin{array}{ccc}0 & 0 & 0\\a &0 & 0\\ 0 & 0 & 0\\\end{array}\right)$$
      where $\delta=0$ or $1$ and $a\neq 0,$
      \item  $$S=\left(
      \begin{array}{ccc}\alpha & \delta & 0\\0 & \alpha & 0\\0 & 0 &  q\alpha\end{array}\right),\quad T=\left(
      \begin{array}{ccc}0 & 0 & 0\\0 &0 & 0\\ 0 & a & 0\\\end{array}\right)$$
      where $\delta=0$ or $1$ and $a\neq 0,$
 \item  $$S=\left(
      \begin{array}{ccc}\alpha & 0 & 0\\0 &  q\alpha & 0\\0 & 0 &  q^2\alpha\end{array}\right),\quad T=\left(
      \begin{array}{ccc}0 & 0 & 0\\a &0 & 0\\ 0 & b & 0\\\end{array}\right)$$
      where $a\neq 0$ or $b\neq 0.$
      \end{enumerate}
      If $ q=-1,$ we have a more general form in (5) and (6):
   \begin{enumerate}[(5.1)]
   \item  $$S=\left(
      \begin{array}{ccc}\alpha & 0 & 0\\0 &  q\alpha & 1\\0 & 0 &  q\alpha\end{array}\right),\quad T=\left(
      \begin{array}{ccc}0 & 0 & b\\a &0 & 0\\ 0 & 0 & 0\\\end{array}\right),$$
      where  $a\neq 0$ or $b\neq 0$,
      \item  $$S=\left(
      \begin{array}{ccc}\alpha & 0 & 0\\0 &  q\alpha & 0\\0 & 0 &  q\alpha\end{array}\right),\quad T=\left(
      \begin{array}{ccc}0 & b & c\\a &0 & 0\\ 0 & 0 & 0\\\end{array}\right),$$
      where  $a\neq 0$ or $b\neq 0$ or $c\neq 0$,
      \end{enumerate}
      \begin{enumerate}[(6.1)]
      \item  $$S=\left(
      \begin{array}{ccc}\alpha & 1 & 0\\0 & \alpha & 0\\0 & 0 &  q\alpha\end{array}\right),\quad T=\left(
      \begin{array}{ccc}0 & 0 & b\\0 &0 & 0\\ 0 & a & 0\\\end{array}\right),$$
      where  $a\neq 0$ or $b\neq 0$,
      \item  $$S=\left(
      \begin{array}{ccc}\alpha & 0 & 0\\0 & \alpha & 0\\0 & 0 &  q\alpha\end{array}\right),\quad T=\left(
      \begin{array}{ccc}0 & 0 & c\\0 &0 & b\\ 0 & a & 0\\\end{array}\right),$$
      where  $a\neq 0$ or $b\neq 0$ or $c\neq 0$.
      \end{enumerate}
   If $ q^3=1,$   we have a more general form in (7)
   $$S=\left(
      \begin{array}{ccc}\alpha & 0 & 0\\0 &  q\alpha & 0\\0 & 0 &  q^2\alpha\end{array}\right),\quad T=\left(
      \begin{array}{ccc}0 & 0 & c\\a &0 & 0\\ 0 & b & 0\\\end{array}\right),$$
      where $a\neq 0$ or $b\neq 0$ or $c\neq 0$.
 \end{remark}
 \begin{proof}
If there is $\alpha,$ an eigenvalue of $S,$ such that neither $\bar q\alpha$ nor $ q\alpha$ is an eigenvalue of $S$ (in particular, $\alpha\neq 0$), then by Lemma \ref{re1}, $T$  maps $H_{S,\alpha}$ to trivial space. However, since $ST\neq 0$, and since a non-trivially twisted commuting pair requires the underlying space of dimension at least two, we get $\dim(H\ominus H_{S,\alpha})= 2$ and $\dim(H_{S,\alpha})= 1$, and in turn, the case (1).

 If for each eigenvalue $\alpha$ of $S$ at least one of $\bar q\alpha$ or $q\alpha$ is an eigenvalue of $S$ as well, then one of the following holds:
  \begin{itemize}
  \item $S$ has an eigenvalue $0$ and two different eigenvalues, which we may assume to be $\alpha$ and $q\alpha$.
  \item $S$ has precisely one eigenvalue $0$.
  \item $S$ has precisely two different eigenvalues, which we may assume to be $\alpha$ and $q\alpha$.
  \item $S$ has three different eigenvalues, which we may assume to be $\alpha,  q\alpha$ and $q^2\alpha.$
  \end{itemize}

 If $S$ has different eigenvalues $0,\alpha, q\alpha,$ then we get the case (2). More precisely, since $\alpha\neq 0$, the pair $(S_1, T_1)$ has the second or the third form in Remark \ref{C2}. The case (2) with $(S_1, T_1)$ of the first form in Remark \ref{C2} is also possible. However, it is the case of $S$ having only $0$ as an eigenvalue, and it is precisely (4) with $d=c=0$.

 If $S$ has the only eigenvalue $0$, then, since $ST\neq 0$ implies $S\neq 0,$ the Jordan form of $S$ matrix can have at most two Jordan blocks. This leads to (3) and (4) respectively, where the formula for $T$ may be verified by a direct calculation.

The remaining cases follow directly from Corollary \ref{model}. Let us only point out that for cases (5) and (6) and $\delta=0$ for general $q,$ one can check $T=\left(
      \begin{array}{ccc}0 & 0 & 0\\0 &0 & 0\\ a & b & 0\\\end{array}\right),$ or $T=\left(
      \begin{array}{ccc}0 & 0 & 0\\a &0 & 0\\ b & 0 & 0\\\end{array}\right)$ to be correct as well, respectively. However, it is not a more general form. Indeed, by the respective choice of eigenvectors for the basis, we get the form as in the statement. In the cases (5.2),(6.2), a similar choice of basis is made to obtain only one nontrivial value $a$ under the diagonal. Such a choice of basis does not guarantee only one nontrivial value above diagonal. In other words, the reduction of nontrivial entries may not be possible above and below the diagonal  simultaneously.
    \end{proof}

\section{Twisted commuting isometries}
This section is devoted to twisted commuting isometries. According to \cite{Majee}, a twisted commuting pair of isometries decomposes into a pair of unitaries and a completely non-unitary pair. The latter one is described by the model \eqref{model}. For a pair of twisted commuting unitaries, the model is known in the finite-dimensional case, where the pair decomposes among $q$-commuting unitaries - see \eqref{mod}.
We are looking for the general model of twisted commuting pairs of unitaries. We have two aims: the first one is  to construct $U$-commuting pairs of unitaries for a given twist $U$; the second one is to describe equivalent conditions for the twisted commutativity of a given a pair of unitaries. Note that, if a pair of unitaries $(V,W)$ is twisted commuting, then the twist is $V^*W^*VW.$

The first aim. The significant  limitation on pairs of $U$ commuting unitaries is the restriction to the commutant of $U$. We know from previous part, that the commutant of some unitaries is  a commutative algebra; for some others, the algebra is not commutative, but there are no twisted commuting isometries. In other words, the pair of $U$-commuting unitaries may not exist. Recall also Corollary \ref{sp_unit}, where twisted commuting unitaries have certain properties. Let us describe the construction of $U$-commuting pair of unitaries $(V,W)$ for a given unitary $U.$ The method uses the fact that $U=V^*W^*VW$ to obtain $(V,W)$ via factorization of $U.$ The procedure and required properties of the factorization are described below.

\begin{proposition}\label{U_mod}
There is a one-to-one correspondence between pairs of commuting unitaries $(A,B)$ intertwined by the unitary $C\in\{AB^*\}'$ and twisted commuting unitaries $(V,W).$
The correspondence is given by the following relations:
\begin{enumerate}
\item Given $(V,W)$ we take $A=VW,\;B=WV,\; C=V^*,$
\item Given $(A, B)$ and $C$ we take $V=C^*,\; W=CA.$
\end{enumerate}

Moreover, the twist is equal to $AB^*=V^*W^*VW$.
\end{proposition}
\begin{proof}
If $(V,W)$ is twisted commuting, $A$ commutes with $B$ by Proposition \ref{gen_twisted} (6). Moreover, the relation $VW=UWV$ yields $U=VW V^*W^*=AB^*$ and, in turn, $V\in\{U\}'$ yields $C\in\{AB^*\}'$ by Fuglede-Putnam theorem. The intertwining property is obvious.

For the reverse implication, let $A,B$ and $C$  be given as in the statement.  Then $V=C^*$ and $W=CA$ commute with $AB^*$ by Fuglede-Putnam theorem and $$VW=C^*C A=A=AB^*B=AB^*(BCC^*)=AB^*(CAC^*)=AB^*WV.$$
\end{proof}

For a fixed unitary $U$, using the relation $U=AB^*,$ we may reduce the set of invariants from the triple $(A,B,C)$ to the pair $(A,C).$
\begin{corollary}
Let $U$ be a unitary operator. There is a one-to-one correspondence between $U$-commuting pairs of unitaries and pairs of unitaries $(A,C)$ satisfying: $A,C\in\{U\}'$ and $U=AC^*A^*C.$
\end{corollary}
\begin{proof}
The result follows from  Proposition \ref{U_mod} with $B=C^*AC.$ Indeed, $U=AC^*A^*C$ yields $B=C^*AC=U^*A\in\{A\}'.$
\end{proof}

The commutant of any unitary is known, so candidates $A,C\in\{U\}'$ are described for any unitary. However, the condition $U=AC^*A^*C$ is not simple to satisfy. The subject is easier to investigate with the help of the operator $B$. Since the twist satisfies $U=AB^*$, for a fixed $U,$ we take a unitary $A\in\{U\}'$ and define $B=U^*A,$ which commutes with $A$ and $U,$ and the pair $(A,B^*)$ factorizes $U$. The intertwining property between $A$ and $B$ by a unitary operator $C$ is equivalent to the unitary equivalence $B=C^*AC.$ Obviously, not all factorizations are between unitary equivalent factors. The trivial example is $U=IU,$ where $1$ is not an eigenvalue of $U$. Moreover, in the case  of factorization between unitary equivalent factors, the operator of unitary equivalence does not necessarily commute with the twist - see Example \ref{e_c}. In fact, we expect $C\notin\{A\}'$ and $C\in\{AB^*\}'.$ Indeed, if $C\in\{A\}',$ then $A=B,$ and in turn, the corresponding pair $(V,W)$ is commuting.

\begin{e}\label{e_c}
Let $$A=\left(\begin{array}{cc} M_z & 0\\0 & I\end{array}\right),\; B=\left(\begin{array}{cc} I & 0\\0 & M_z\end{array}\right),\text{ and so } AB^*=\left(\begin{array}{cc} M_z & 0\\0 & M_{\bar{z}}\end{array}\right)$$
be operators on $L^2(\mathbb{T})\oplus L^2(\mathbb{T}).$
One may check, that the unitary operator intertwining $A$ and $B$ has the form $C=\left(\begin{array}{cc} 0 & X\\Y & 0\end{array}\right),$ where $X,Y$ are unitary and $Y\in\{M_z\}'.$ However, $C\in\{AB^*\}'$ yields a contradiction $M_{\bar{z}}Y=YM_z.$
\end{e}

In Example \ref{e_c}, the factorization between unitarily equivalent factors does not generate a pair of twisted commuting unitaries. The following result describes all $U$-commuting pairs corresponding to one factorization of $U$.

\begin{theorem}\label{co_u}
Let $U$ be a unitary and $(A,B)$ be unitarily equivalent unitaries, such that $U=AB^*.$
Let $\mathfrak{C}_{A,U}$ denotes the set of $U$-commuting pairs of unitaries $(V,W)$ such that $VW=A, WV=B.$

Either $\mathfrak{C}_{A,U}=\emptyset$ or there is a one-to-one map $\{A,U\}'\mapsto\mathfrak{C}_{A,U}.$
\end{theorem}
\begin{proof}
By Example \ref{e_c}, $\mathfrak{C}_{A,U}$ may be empty.

Assume $(V_1,W_1)\in\mathfrak{C}_{A,U},$ so the set is not empty. Then, taking $C_1=V_1^*$ in Proposition \ref{U_mod},  we have $V_1^*A=BV_1^*.$  If there is another pair $(V_2,W_2)\in \mathfrak{C}_{A,U},$ then $V_2^*A=BV_2^*.$ Consequently, $V_1AV_1^*=B=V_2AV_2^*,$ and in turn $V_2^*V_1\in\{A\}'.$ Since $V_1,V_2\in\{U\}',$ also $V_1^*V_2\in\{U\}'.$  Note that $(V_1,W_1)\neq (V_2,W_2),$ by $V_1W_1=A=V_2W_2$ yields $V_1\neq V_2.$ Since all operators are unitary, the map $\mathfrak{C}_{U,A}\ni (V,W)\mapsto V_1^*V\in \{A,U\}'$ is injective. On the other hand, for any $D\in\{A,U\}',$ the pair $(V_1D,D^*W_1)\in \mathfrak{C}_{U,A}.$ Indeed, $V_1D, D^*W_1$ commute with $U$ as products of such operators. Moreover, $V_1DD^*W_1=V_1W_1=A$ and $D^*W_1V_1D=D^*BD=D^*U^*AD=D^*DU^*A=B.$
\end{proof}

The second aim is to give an equivalent condition for twisted commutiativity between unitaries. As in the commuting case,  twisted commuting unitaries are doubly twisted commuting. Moreover, Proposition \ref{gen_dtwisted}(3) in the case of general isometries is correct under a weaker assumption of twisted commutativity. Let us give an equivalent condition for the pair of isometries to be twisted commuting.

\begin{proposition}\label{iso_twist}
A pair of isometries $(V,W)$ is twisted commuting if and only if the following conditions hold:
\begin{itemize}
\item $\ran{VW}=\ran{WV},$
\item $W^*VW$ commutes with $V,$
\item $V^*WV$ commutes with $W$.
\end{itemize}
\end{proposition}
\begin{proof}
 Since $VW=WV(V^*W^*VW),$ the pair $(V,W)$ is twisted commuting if and only if $V^*W^*VW$ is unitary and it commutes with $V$ and $W$. Clearly, $V^*W^*VW$ is unitary if and only if $\ran{VW}=\ran{WV}.$

Assume that $V^*W^*VW$ is unitary. Since unitary is normal, if $V$ commutes with $W^*VW,$ then it commutes with $V^*W^*VW$. For the reverse implication, note that $$V(V^*W^*VWV-VV^*W^*VW)=W^*VWV-VW^*VW.$$

Since $V^*W^*VW$ is assumed to be unitary, by Fuglede-Putman theorem, it commutes with $W$ if and only if its adjoint $(V^*W^*VW)^*=W^*V^*WV$ commutes with $W$. Hence, by similar arguments as above, we get that $W^*V^*WV$ commutes with $W$ if and only if $V^*WV$ commutes with $W$.

\end{proof}

In the case of unitary operators the first condition in Proposition \ref{iso_twist} is trivially satisfied. Any of the remaining two conditions may be replaced by the commutativity of products.

\begin{remark}\label{re_u}
Let $(V,W)$ be a pair of unitaries. Any two of the following conditions implies the third one:
\begin{enumerate}
\item $WV$ commutes with $VW$,
\item $W^*VW$ commutes with $V,$
\item $V^*WV$ commutes with $W.$
\end{enumerate}

Let us show that (2) and (3) imply (1). By Fuglede-Putnam theorem, (2) and (3) yield that $U:=W^*V^*WV$ commutes with $W$ and $V.$ Hence, $VWWV=(WVU)WV=WV(UWV)=WV(WVU)=WVVW$.

Let us show that (1) and (2) imply (3). By Fuglede-Putnam theorem, $V^*W^*$ commutes with $VW$ and, in turn,  \begin{align*}0=V^{*2}(W^*VWV-VW^*VW)=V^*V^*W^*VWV-V^*W^*VW\\=V^*VWV^*W^*V-V^*W^*VW=WV^*W^*V-V^*W^*VW.\end{align*}
Clearly, (1) and (3) imply (2) in a similar way.
\end{remark}
Hence, conditions Proposition \ref{gen_twisted}(6) and Proposition \ref{gen_dtwisted}(3) are sufficient for the twisted commutativity of unitaries.
\begin{theorem}\label{iso_unitary}
A pair of unitaries $(V,W)$ is twisted commuting if and only if their products commute and $W^*VW$ commutes with $V,$ (or equivalently $V^*WV$ commutes with $W$.)
\end{theorem}

Conditions in Proposition \ref{iso_unitary} are optimal, as illustrated by the following examples.

\begin{e}\label{ex_a}
Let $V=M_z$ and $W=M_{z^2}(P+i(I-P))$  be two unitaries on $H=L^2(\mathbb{T})$, where $P$ is an orthogonal projection onto $\sspan{z^{2n}:n\in\mathbb{Z}}.$

Note that $M_{z^n}$ commutes with $P$ for $n$ odd, while $M_{z^n}P=(I-P)M_{z^n}$ and $M_{z^n}(I-P)=PM_{z^n}$ for $n$ even. In other words, $P+i(I-P)$ commutes with $M_{z^n}$ for $n$ odd, and  $M_{z^n}(P+i(I-P))=(iP+(I-P))M_{z^n}$ for $n$ even. In particular, $P$ commutes with $W$ and  $VP=(I-P)V,\; V(I-P)=PV$. Thus, we get  \begin{align}\label{exx}\begin{split}VW=M_{z^3}(P+i(I-P))=(iP+(I-P))M_{z^3},\\WV=M_{z^2}(P+i(I-P))M_z=(P +i(I-P))M_{z^3}.\end{split}
\end{align}
Hence, $VW=UWV$ for $U:=i(P-(I-P))=i(2P-I),$ which is unitary and commutes with $W.$ However, $UV=-VU$ and so, the pair $(V,W)$ is not twisted commuting.

Let us find unsatisfied conditions in Proposition \ref{re_u}. Since $$V^*WV=M_z^*M_{z^2}(P+i(I-P))M_z=M_{z^2}(iP+(I-P))$$ we get $V^*WVW=WV^*WV=iM_{z^4}$, so $V^*WV$ commutes with $W.$ Thus, the other two conditions are not satisfied. Indeed,  \begin{align*}W^*VW&=(P-i(I-P))M_{z^2}^*M_zM_{z^2}(P+i(I-P))\\&=(P-i(I-P))M_z(P+i(I-P))=M_zi(I-2P)\end{align*} yields $$W^*VWV=M_zi(I-2P)M_z=M_{z^2}i(2P-I)\neq M_{z^2}i(I-2P)=VW^*VW.$$ Also, by \eqref{exx}, we get \begin{align*}
WVVW=(P +i(I-P))M_{z^3}(iP+(I-P))M_{z^3}=(P-(I-P))M_{z^6}\\\neq((I-P)-P)M_{z^6}=(iP+(I-P))M_{z^3}(P +i(I-P))M_{z^3}=VWWV.\end{align*}
\end{e}

\begin{e}\label{ex_b}
Let $V=M_z$ and $W=M_zP+M_{\bar{z}}(I-P)$  be two unitaries on $H=L^2(\mathbb{T})$, where $P$ is an orthogonal projection onto $\sspan{z^{2n}:n\in\mathbb{Z}}$.
Then $VW=M_{z^2}P+(I-P)$ commutes with $WV=M_{z^2}(I-P)+P.$ However, neither $W^*VW=(PM_{\bar{z}}+(I-P)M_z)M_z(M_zP+M_{\bar{z}}(I-P))=M_{z^3}P+M_{\bar{z}}(I-P)$ does commute with $V$ nor $V^*WV=P+M_{\bar{z}^2}(I-P)$ commutes with $W$. Hence, the pair is not twisted commuting. Indeed, the twist $V^*W^*VW=(I-P)M_{\bar{z}^2}+M_{z^2}P=M_{\bar{z}^2}(I-P)+M_{z^2}P$ does not commute either with $V$ or $W$.
\end{e}

\end{document}